\documentclass[a4paper]{amsart}

\usepackage{amsfonts}
\usepackage{amssymb}
\usepackage{amsmath}
\usepackage{hyperref}
\usepackage{mathrsfs}
\usepackage{centernot}
\usepackage{mathdots}
\usepackage{stmaryrd}
\usepackage[all]{xy}

\newtheorem{thm}{Theorem}[section]

\newtheorem{baseexample}[thm]{Example} 
\newtheorem{baseremark}[thm]{Remark} 

\newcommand{\rem}[1]{}

\newcommand{\Step}[1]{\noindent {\bf Step #1.}} 

\newcommand{\F}{\mathbb{F}}

\newcommand{\Q}{\mathbb{Q}}

\newcommand{\frakSmall}{
\newcommand{\fraka}{{\mathfrak{a}}}
\newcommand{\frakb}{{\mathfrak{b}}}
\newcommand{\frakc}{{\mathfrak{c}}}
\newcommand{\frakd}{{\mathfrak{d}}}
\newcommand{\frake}{{\mathfrak{e}}}
\newcommand{\frakf}{{\mathfrak{f}}}
\newcommand{\frakg}{{\mathfrak{g}}}
\newcommand{\frakh}{{\mathfrak{h}}}
\newcommand{\fraki}{{\mathfrak{i}}}
\newcommand{\frakj}{{\mathfrak{j}}}
\newcommand{\frakk}{{\mathfrak{k}}}
\newcommand{\frakl}{{\mathfrak{l}}}
\newcommand{\frakm}{{\mathfrak{m}}}
\newcommand{\frakn}{{\mathfrak{n}}}
\newcommand{\frako}{{\mathfrak{o}}}
\newcommand{\frakp}{{\mathfrak{p}}}
\newcommand{\frakq}{{\mathfrak{q}}}
\newcommand{\frakr}{{\mathfrak{r}}}
\newcommand{\fraks}{{\mathfrak{s}}}
\newcommand{\frakt}{{\mathfrak{t}}}
\newcommand{\fraku}{{\mathfrak{u}}}
\newcommand{\frakv}{{\mathfrak{v}}}
\newcommand{\frakw}{{\mathfrak{w}}}
\newcommand{\frakx}{{\mathfrak{x}}}
\newcommand{\fraky}{{\mathfrak{y}}}
\newcommand{\frakz}{{\mathfrak{z}}}
}

\newcommand{\calCapital}{
\newcommand{\calA}{{\mathcal{A}}}
\newcommand{\calB}{{\mathcal{B}}}
\newcommand{\calC}{{\mathcal{C}}}
\newcommand{\calD}{{\mathcal{D}}}
\newcommand{\calE}{{\mathcal{E}}}
\newcommand{\calF}{{\mathcal{F}}}
\newcommand{\calG}{{\mathcal{G}}}
\newcommand{\calH}{{\mathcal{H}}}
\newcommand{\calI}{{\mathcal{I}}}
\newcommand{\calJ}{{\mathcal{J}}}
\newcommand{\calK}{{\mathcal{K}}}
\newcommand{\calL}{{\mathcal{L}}}
\newcommand{\calM}{{\mathcal{M}}}
\newcommand{\calN}{{\mathcal{N}}}
\newcommand{\calO}{{\mathcal{O}}}
\newcommand{\calP}{{\mathcal{P}}}
\newcommand{\calQ}{{\mathcal{Q}}}
\newcommand{\calR}{{\mathcal{R}}}
\newcommand{\calS}{{\mathcal{S}}}
\newcommand{\calT}{{\mathcal{T}}}
\newcommand{\calU}{{\mathcal{U}}}
\newcommand{\calV}{{\mathcal{V}}}
\newcommand{\calW}{{\mathcal{W}}}
\newcommand{\calX}{{\mathcal{X}}}
\newcommand{\calY}{{\mathcal{Y}}}
\newcommand{\calZ}{{\mathcal{Z}}}
}

\newcommand{\bbCapital}{
\newcommand{\bbA}{{\mathbb{A}}}
\newcommand{\bbB}{{\mathbb{B}}}
\newcommand{\bbC}{{\mathbb{C}}}
\newcommand{\bbD}{{\mathbb{D}}}
\newcommand{\bbE}{{\mathbb{E}}}
\newcommand{\bbF}{{\mathbb{F}}}
\newcommand{\bbG}{{\mathbb{G}}}
\newcommand{\bbH}{{\mathbb{H}}}
\newcommand{\bbI}{{\mathbb{I}}}
\newcommand{\bbJ}{{\mathbb{J}}}
\newcommand{\bbK}{{\mathbb{K}}}
\newcommand{\bbL}{{\mathbb{L}}}
\newcommand{\bbM}{{\mathbb{M}}}
\newcommand{\bbN}{{\mathbb{N}}}
\newcommand{\bbO}{{\mathbb{O}}}
\newcommand{\bbP}{{\mathbb{P}}}
\newcommand{\bbQ}{{\mathbb{Q}}}
\newcommand{\bbR}{{\mathbb{R}}}
\newcommand{\bbS}{{\mathbb{S}}}
\newcommand{\bbT}{{\mathbb{T}}}
\newcommand{\bbU}{{\mathbb{U}}}
\newcommand{\bbV}{{\mathbb{V}}}
\newcommand{\bbW}{{\mathbb{W}}}
\newcommand{\bbX}{{\mathbb{X}}}
\newcommand{\bbY}{{\mathbb{Y}}}
\newcommand{\bbZ}{{\mathbb{Z}}}
}

\newcommand{\veps}{\varepsilon}

\newcommand{\onto}{\twoheadrightarrow}

\newcommand{\suchthat}{\,:\,}

\newcommand{\SMatII}[4]{\left[\begin{array}{cc} {#1} & {#2} \\ {#3} &
{#4} \end{array}\right]}
\newcommand{\smallSMatII}[4]{\left[\begin{smallmatrix} {#1} & {#2} \\ {#3} &
{#4} \end{smallmatrix}\right]}
\newcommand{\SMatIII}[9]{\left[\begin{array}{ccc} {#1} & {#2} & {#3} \\ {#4} &
{#5} & {#6} \\ {#7} & {#8} & {#9} \end{array}\right]}

\DeclareMathOperator{\Gal}{Gal} %
\DeclareMathOperator{\Hom}{Hom} %
\DeclareMathOperator{\im}{im} %
\DeclareMathOperator{\Jac}{Jac} %

\newcommand{\nGL}[2]{\mathrm{GL}_{#2}({#1})}

\newcommand{\nMat}[2]{\mathrm{M}_{#2}(#1)}

\newcommand{\units}[1]{{#1^\times}}

\newtheorem{theorem}{Theorem}
\newtheorem{proposition}[theorem]{Proposition}

\newtheorem{mybaseremark}[theorem]{Remark}

\newenvironment{myremark}
{\begin{mybaseremark}\rm}{\end{mybaseremark}}

\calCapital %
\frakSmall %
\bbCapital %

\title[Rationally Isometric Quadratic Forms]{An Elementary Proof That Rationally Isometric Quadratic
Forms Are Isometric}

\author{Uriya A.\ First}
\date{\today}
\address{Einstein Institute of Mathematics, Hebrew University of Jerusalem}
\email{uriya.first@gmail.com}
\thanks{This research was supported by a Swiss National Foundation of Science Grant no.\ IZK0Z2\_151061}

\begin{document}

\maketitle

\begin{abstract}
    Let $R$ be a valuation ring with fraction field $K$ and $2\in R^\times$.
    We give an elementary proof of the following known result: Two  unimodular quadratic forms over $R$ are isometric over $K$ if and only
    if they are isometric over $R$. Our proof does not use cancelation of quadratic forms and yields an
    explicit algorithm to construct an isometry over $R$ from a given isometry over $K$.
    The statement actually holds for hermitian forms over valuated involutary division rings, provided mild assumptions.
\end{abstract}

\bigskip

Let $S/R$ be a separable commutative ring extension with $R$ a local integral domain, let $F$ be the fraction
field of $R$, and let $A$ be an Azumaya
$S$-algebra admitting an involution $\sigma$ with $R=S^\sigma:=\{s\in S\suchthat s^\sigma=s\}$.
Assume that $2\in\units{R}$.
It was shown in \cite{OjanPanin01} that when $R$ is regular and contains a field,
a unimodular hermitian form over $A$ is hyperbolic over $A\otimes_RF$
if and only if it is hyperbolic over $A$. This is a special case of the Grothendieck-Serre conjecture;
see for instance \cite[\S1]{Panin05} and related papers.
The same statement was established in \cite[Pr.\ 3.6]{Beke13} when $R$ is a valuation ring (or
even an intersection of finitely many) and $A$ has no zero-divisors. By invoking a cancelation theorem of Keller
(\cite[Th.\ 3.4.2]{Keller88}) as done in \cite[Pr.\ 2.14]{BekeVGeel14},
this implies that when $R$ is a valuation ring or local regular ring containing a field,
two unimodular
hermitian forms over $(A,\sigma)$ are isometric over $A\otimes_R F$ if and only if
they are isometric over $R$. Special cases of these results were proved much earlier by various authors
(e.g.\ see \cite[\S6.2]{SchQuadraticAndHermitianForms} and similar references).

\medskip

We give  an elementary proof of the following special case:

\begin{theorem}\label{TH}
    Let $K/F$ be a field extension admitting an involution $\sigma\in\Gal(K/F)$
    such that $K^\sigma=F$, and let $\nu:\units{K}\onto \Gamma$ be an (additive) valuation with $\nu\circ\sigma=\nu$
    and $\nu(2)=0$.
    Denote by $S$ and $R$ the valuation rings of $\nu$ in $K$ and $F$, respectively.
    If $K/F$ is unramified with respect to $\nu$ (i.e.\ $\im(\nu|_F)=\Gamma$), then
    two unimodular $1$-hermitian forms over $(S,\sigma)$ are isometric over $K\cong S\otimes_R F$ if
    and only if they are isometric over $S$.
\end{theorem}

Our proof is elementary and avoids cancelation of hermitian forms over local rings.
Furthermore,
it yields an explicit  algorithm to construct an isometry over $S$ from an isometry over $K$.
(The basic operations required for the algorithm are arithmetic operations in $K$ and $\Gamma$, applying $\nu$,
and producing elements of $F$ with a given valuation.)

\medskip

Before giving the proof, let us recall some of the definitions; see \cite{SchQuadraticAndHermitianForms}
for an extensive discussion: Let $\veps=\pm 1$. An \emph{$\veps$-hermitian space} over a ring with
involution $(S,\sigma)$ consists of a pair $(M,h)$ where  $M$ is a projective left $S$-module, and
$h:M\times M\to S$ is a biadditive map satisfying $h(ax,by)=a h(x,y)b^\sigma$ and $h(x,y)=\veps h(y,x)^\sigma$ for all $x,y\in M$ and
$a,b\in M$. In this case, $h$ is called an \emph{$\veps$-hermitian form}. We say that $h$ is \emph{unimodular}
if the map $\tilde{h}:M\to \Hom_S(M,S)$ given by $\tilde{h}(x)= [y\mapsto h(y,x)]$ is an isomorphism.
Two hermitian spaces $(M,h)$ and $(M',h')$ are called \emph{isometric} if there is an isomorphism
$u:M\to M'$ such that $h'(ux,uy)=h(x,y)$.  If $(K,\tau)$ is a ring with involution containing $(S,\sigma)$,
then $(M,h)$ gives rise to an $\veps$-hermitian space $(M_K,h_K)$ over $(K,\tau)$
defined by $M_K=K\otimes_S M$ and $h_K(a\otimes x,b\otimes y)=a h(x,y) b^\tau$ ($x,y\in M$, $a,b\in K$).
We say that $(M,h)$ and $(M',h')$ are isometric over $(K,\tau)$ if
$(M_K, h_K)$ and $(M'_K,h'_K)$ are isometric.

\begin{proof}[Proof of Thoerem~\ref{TH}]

For a matrix $a=(a_{ij})_{i,j}$ over $K$, write $a^*:=(a_{ji}^\sigma)_{i,j}$.
We say that $a$ is $*$-symmetric if $a^*=a$.

\medskip

\Step{0}
Assume $(M,h)$ and $(M',h')$ are unimodular $1$-hermitian spaces over $S$ that are isometric
over $K$. Then $\dim_S M=\dim_S M'$, hence we may assume $M=M'=S^n$ (viewed as row vectors).
In this case, there are unique $*$-symmetric matrices $a,b\in\nGL{S}{n}$
such that $h(u,v)=uav^*$ and $h'(u,v)=ubv^*$ for all $u,v\in S^n$.
It is well-known that $h\cong h'$ over $K$ (resp.\ $S$) if and only if
there exists $u\in\nGL{K}{n}$ (resp.\
$u\in\nGL{S}{n}$) such that $uau^*=b$.
It is thus enough to show that if $uau^*=b$ with $u\in\nGL{K}{n}$, then
$u$ can be taken to be in $\nGL{S}{n}$.

\medskip

\Step{1} We may assume $u$ is diagonal. Indeed, it is a standard claim that in every B\'{e}zout domain (i.e.\ a domain
whose f.g.\ ideals are principle), $S$ in particular,
\begin{equation}\label{EQ:one}
\nGL{K}{n}=\nGL{S}{n}\cdot T\cdot \nGL{S}{n},
\end{equation}
where $T$ denotes the  diagonal matrices in $\nGL{K}{n}$.
For the sake of completeness, we shall recall the proof for valuation rings in Proposition~\ref{PR}.
Now, using \eqref{EQ:one}, write $u=xu'y$ with $x,y\in \nGL{S}{n}$ and $u'\in T$.
We may replace
$a,b,u$ with $yay^*,x^{-1}b(x^{-1})^*,u'$.

\medskip

\Step{2} Let $I_r$ denote the identity matrix
of size $r$. We claim that we may assume $u$ is of the form $\pi I_r\oplus \pi^{-1} I_r\oplus I_{n-2r}$
for some $\pi\in \Jac(R)$ and $r>0$.

For a matrix $x=(x_{ij})_{i,j}$, write $\nu(x)=(\nu(x_{ij}))_{i,j}$. Then $\nu(x)$
is a matrix with entries in $\Gamma\cup\{\infty\}$ (where $\nu(0)=\infty$).
Let $u_1,\dots,u_n$ be the diagonal entries of $u$ and
let $\gamma_t>\gamma_{t-1}>\dots > \gamma_1>\gamma_0=0$
be the  absolute values of the valuations of $u_1,\dots,u_n$ together with $0\in \Gamma$.
Conjugating $a,b,u$ by a suitable permutation matrix, we may
assume that
\[\nu(u_1)=\dots=\nu(u_r)=\gamma_t, \qquad \nu(u_{r+1})=\dots=\nu(u_{r+s})=-\gamma_t,\]
and $|\nu(u_i)|<\gamma_t$ for all $i>r+s$.

We  claim that $r=s$. Indeed,
write $(\alpha_{ij})=\nu(a)$, $(\beta_{ij})=\nu(b)$ and $\tau_i=\nu(u_i)=\nu(u_i^\sigma)$.
Then
\[
(\beta_{ij})_{i,j}=\nu(uau^*)=(\alpha_{ij}+\tau_i+\tau_j)_{i,j},
\]
hence $\alpha_{ij}+\tau_i+\tau_j\geq 0$ for all $i,j$.
Since $\tau_i+\tau_j<0$ when $r<i\leq r+s$ and $r<j$, we have
\[\nu(a_{ij})>0\qquad\forall\quad r<i\leq r+s,\quad r<j\ .\]
If $s>r$,  this implies that rows $r+1,\dots,r+s$ of $a$ are linearly dependant over
$S/\Jac(S)$, contrary to our assumption that $a$ is invertible over $S$.
Thus, $s\leq r$.
Applying a similar argument to show that $\nu(\beta_{ij})>0$ when $i\leq r$ and $j\leq r$ or $r+s<j$ yields that
$s\geq r$, so $s=r$.

We now apply induction to $t$, the case $t=0$ being clear.
Pick an element $\pi\in R$ of valuation $\gamma_t-\gamma_{t-1}$
(here we use the assumption that $K/F$ is unramified),
let $u'=\pi I_r\oplus \pi^{-1} I_r\oplus I_{n-2r}$,
and write $b'=u'au'^*$ and $(\beta'_{ij})=\nu(u'au'^*)$.
We claim that $\beta'_{ij}\geq 0$ for all $i$. As $b'$ is $*$-symmetric,
we only need to verify this for $i\leq j$. Indeed, we have
\[
\beta'_{ij}=\left\{\begin{array}{ll}
\alpha_{ij}+2\gamma_t-2\gamma_{t-1} & i\leq j\leq r\\
\alpha_{ij}-2\gamma_t+2\gamma_{t-1} & r<i\leq j\leq 2r\\
\alpha_{ij} & \text{$i\leq r<j\leq 2r$ or $2r<i\leq j$}\\
\alpha_{ij}+\gamma_t-\gamma_{t-1} & i\leq r<2r<j\\
\alpha_{ij}-\gamma_t+\gamma_{t-1} & r<i\leq 2r<j
\end{array}\right.
\]
All the cases are clear except the second and the fifth.
Recall that $\alpha_{ij}+\tau_i+\tau_j=\beta_{ij}$.
Now,
in
second case, we have
\[\alpha_{ij}-2\gamma_t+2\gamma_{t-1}=\alpha_{ij}+\tau_i+\tau_j+2\gamma_{t-1}=\beta_{ij}+2\gamma_{t-1}\geq 0,\]
and in the fifth case, writing $|\tau_j|=\gamma_s$ for $s<t$ (and noting that $\tau_i=-\gamma_t$), we get
\[\alpha_{ij}-\gamma_t+\gamma_{t-1}=\beta_{ij}-\tau_j+\gamma_{t-1}\geq \beta_{ij}+\gamma_{t-1}-\gamma_s\geq 0,\]
as required.
Thus, $b'\in\nMat{S}{n}$. Since $\nu(\det(b'))=\nu(\det(u'au'^*))=\nu(\det(a))=0$ (because $\det(u')=1$),
we have $b'\in\nGL{S}{n}$.
Therefore, we may replace $a,b,u$ with $a,b',u'$
and apply induction to $b',b,uu'^{-1}$ (the parameter $t$ is decreased by $1$).

\medskip

\Step{3} We may assume $n=2r$ and $u=\pi I_r\oplus \pi^{-1}I_r$ with $\pi\in\Jac(R)$.
Indeed,
suppose $u=\pi I_r\oplus \pi^{-1}I_r\oplus I_{n-2r}$ as in step 2. Then we may write
\[
a=\SMatIII{a_{11}}{a_{12}}{a_{13}}{a_{12}^*}{\pi^2 a_{22}}{\pi a_{23}}{a_{13}^*}{a_{23}^*\pi}{a_{33}},\qquad
b=\SMatIII{\pi^2 a_{11}}{a_{12}}{\pi a_{13}}{a_{12}^*}{a_{22}}{a_{23}}{a_{13}^*\pi}{a_{23}^*}{a_{33}}
\]
where $a_{11},a_{12},a_{13},a_{22},a_{23},a_{33}$ are  matrices over $S$ and $a_{11},a_{12},a_{22}$
are of size $r\times r$.
Observe that the image of $a$ in $\nMat{S/\Jac(S)}{n}$
has the form
\[
\SMatIII{*}{*}{*}{*}{ 0}{ 0}{*}{0}{*}\ .
\]
This implies that $a_{33}$ is invertible over $S$.
Let
\[
v=\SMatIII{1}{0}{-a_{13}a_{33}^{-1}}{0}{1}{-\pi a_{23}a_{33}^{-1}}{0}{0}{1},\qquad w=\SMatIII{1}{0}{-\pi a_{13}a_{33}^{-1}}{0}{1}{-a_{23}a_{33}^{-1}}{0}{0}{1}
\]
Then $v,w\in\nGL{S}{n}$ and it is easy to check
that
\[
vav^*=
\SMatIII{x}{y}{0}{y^*}{\pi^2 z}{0}{0}{0}{a_{33}}
,\qquad wbw^*=
\SMatIII{\pi^2 x}{y}{0}{y^*}{z}{0}{0}{0}{a_{33}}\ .
\]
where
\[
x=a_{11}-a_{13}a_{33}^{-1}a_{13}^*,\qquad y=a_{12}-a_{13}a_{33}^{-1}a_{23}^*\pi,\qquad z=a_{22}-a_{23}a_{33}^{-1}a_{23}^*\ . 
\]
Therefore, we may replace $a,b,u$ with $\smallSMatII{x}{y}{y^*}{\pi^2 z}$,
$\smallSMatII{\pi^2x}{y}{y^*}{z}$, $\pi I_r\oplus \pi^{-1}I_r$.

\medskip

\Step{4} Using the notation of the previous step,
it is left to show that for any $*$-symmetric $x,z\in\nMat{S}{r}$ and $y\in\nGL{S}{r}$,
there is $u\in\nGL{S}{n}$ such that $u\smallSMatII{x}{y}{y^*}{\pi^2z}u^*=\smallSMatII{\pi^2 x}{y}{y^*}{z}$.
Replacing $a=\smallSMatII{x}{y}{y^*}{\pi^2z}$, $b=\smallSMatII{\pi^2 x}{y}{y^*}{z}$
with $vav^*$, $vbv^*$ for $v =\smallSMatII{y^{-1}}{0}{0}{I_r}$,
we may further assume that $y=I_r$.
Assuming this, let
\[
w=(2-\pi x-\pi z)^{-1}
\]
and take
\[
u=\SMatII{\pi-2\pi(1-\pi x)w}{2(1-\pi x)w}{2(1-\pi z)w}{\pi^{-1}-2\pi^{-1}(1-\pi z)w}\ .
\]
(We have $\pi^{-1}-2\pi^{-1}(1-\pi z)w\in \nMat{S}{r}$ since
$1-2(1-\pi z)w\equiv 0$ in $S/\pi S$.)
The verification of
\begin{equation}\label{EQ:eq-II}
u\SMatII{x}{1}{1}{\pi^2z}u^*=\SMatII{\pi^2 x}{1}{1}{z}
\end{equation} is tedious and left to the reader; the identity
$w(2-\pi x-\pi y)w=w$
is used repeatedly to cancel terms.
\end{proof}

We now recall the known algorithm to decompose
a matrix in $\nGL{K}{n}$ as a product $xuy$ with $x,y\in\nGL{S}{n}$
and diagonal $u\in\nGL{K}{n}$.

\begin{proposition}\label{PR}
    For any $z\in\nGL{K}{n}$, there are  $x,y\in\nGL{S}{n}$
    and diagonal $u\in\nGL{K}{n}$ with $z=xuy$.
\end{proposition}

\begin{proof}
    Write $z=(z_{ij})$. Multiplying $z$ on the left and on
    the right by suitable permutation matrices, we may assume $\nu(z_{11})=\min_{i,j}\nu(z_{ij})$.
    We may now replace $z$ with
    \[
    \left[\begin{smallmatrix}
    1 & & & \\
    -z_{21}z_{11}^{-1} & 1& & \\
    \vdots & & \ddots & \\
    -z_{n1}z_{11}^{-1} & & & 1
    \end{smallmatrix}\right]
    \cdot z\cdot
    \left[\begin{smallmatrix}
    1 & -z_{11}^{-1}z_{12} & \ldots & -z_{11}^{-1}z_{1n} \\
    & 1 & & \\
    & & \ddots & \\
    & & & 1
    \end{smallmatrix}\right]=
    \left[
    \begin{smallmatrix}
    z_{11} &  & & \\
    & * & \cdots & * \\
    & \vdots & & \vdots \\
    & * & \cdots & *
    \end{smallmatrix}
    \right]
    \]
    and proceed by induction on $n$.
\end{proof}

\begin{myremark}
    With some modification, the identity \eqref{EQ:eq-II} holds even without assuming $K$ is commutative
    and $\pi$ is central in $K$ (but $\pi^\sigma=\pi$ is still assumed). Indeed,
    if
    \[
    u=\SMatII{\pi-2(1-\pi x)(2-\pi x-\pi z)^{-1}\pi}{2(1-\pi x)(2-\pi x-\pi z)^{-1}}{
    2(1- z\pi)(2-x\pi-z\pi)^{-1}}{\pi^{-1}-2(1- z\pi)(2-x\pi-z\pi)^{-1}\pi^{-1}
    }\ ,
    \]
    then
    one has
    \[
    u\SMatII{x}{1}{1}{\pi z\pi}u^*=\SMatII{\pi x\pi}{1}{1}{z}\ .
    \]
    Again, the computation is left to the reader.

    This modified identity can be used to prove the following stronger version of
    Theorem~1: Let $K$ be a division ring admitting an additive valuation
    $\nu:\units{K}\onto \Gamma$ ($\Gamma$ is abelian) and an involution $\sigma$ such
    $\nu(2)=0$, $\nu\circ \sigma=\nu$, and for all $\gamma\in\Gamma$ there is $a\in K$
    with $\nu(a)=\gamma$ and $a=a^\sigma$. Let $S:=\{a\in K\suchthat \nu(a)\geq 0\}$.
    Then two unimodular $1$-hermitian forms over $(S,\sigma)$ are isometric over $K$ if and only if they are isometric over $S$.

    The proof  follows the same lines, but requires few additional technicalities, such as verifying
    that $b'=u'au'$ is invertible in $\nMat{S}{n}$ in Step~2 (the claim $\det(b')=\det(a)\in S$ is
    meaningless since $\det(\,\cdot\,)$ is no longer defined). This can be settled either by using
    \emph{Dieudonn\'{e} determinants} (together with the easy fact that
    $\units{S}/[\units{S},\units{S}]\hookrightarrow\units{K}/[\units{K},\units{K}]$;
    see \cite[p.\ 64]{Ros94AlgKThy}),
    or by carefully studying the block form of $a$ and $a^{-1}$ in $\nMat{S/\pi^2S}{n}$.
\end{myremark}

\begin{myremark}
    Reversing the proof of Theorem~\ref{TH} up to Step 2 gives  a recipe for producing
    examples of rationally isometric hermitian forms: Start with a \emph{diagonal} matrix $u\in\nGL{K}{n}$
    such that for every $\gamma\in\Gamma$, the number of diagonal entries with valuation $\gamma$
    is the same as the number of diagonal entries  with valuation $-\gamma$.
    Next, choose a matrix $a\in\nGL{S}{n}$ such that $a^*=a$ and $b:=uau^*\in\nGL{S}{n}$. The choice
    of such $a$ can be easily made rigorous.
    Finally, replace $a$, $b$, $u$ with $xax^*$, $yby^*$, $yux^{-1}$ for some $x,y\in \nGL{S}{n}$. By the proof of Theorem~\ref{TH},
    every example is obtained in this way.
\end{myremark}

\begin{myremark}
    As commented in the beginning,
    it is possible to derive an algorithm from the proof of
    Theorem~\ref{TH}. If hermitian forms and isometries are represented as $n\times n$ matrices over $K$,
    then the operations required are:
    \begin{enumerate}
        \item $O(n)$ multiplications and inversions of matrices over $K$ of
        size at most $n\times n$ (we include here also additions of matrices whose
        complexity is smaller),
        \item $O(n^3)$ applications of $\nu$ ($O(n)$ in Th.~\ref{TH} and $O(n^3)$ in Pr.~\ref{PR}),
        \item $O(n^3)$ elementary elementary operations ($+,-,\leq$) in $\Gamma$ ($O(n^3)$ comparisons in
        Pr.~\ref{PR} and at most $O(n^2)$ operations in Step~2),
        \item $O(n)$ operations of finding perimages for elements of $\Gamma$ in $R$
        (Step~2).
    \end{enumerate}
    
    Assuming a computational model in which elementary operations in $K$ and $\Gamma$ take $O(1)$ time, 
    the complexity of (a) overtakes all other operations, so this is the asymptotic complexity
    of the algorithm. A naive implementation of matrix multiplication and inversion would
    then yield complexity of $O(n^4)$. (More efficient algorithms for matrix 
    multiplication are Strassen and Coppersmith{-}Winograd, for instance.)
    There seems be considerable overlap between the matrices being multiplied
    and inverted through the algorithm, so perhaps further improvement can be achieved by
    cleverly exploiting this.
    
    In practice, since $K$ is infinite, the complexity of elementary operations in elements of $K$ depends
    on the number of bits required to represent the elements, and usually these
    operations yield elements of longer presentation (consider $\Q$ or $\F_p(t_1,\dots,t_n)$ for example).
    Estimating the exact affect of this on the complexity of the algorithm is complicated, and may heavily
    depend on the base field and its presentation.
    However, experiments suggests that the practical worst-case complexity is  exponential
    when $K=\Q$; the output isometry has very large nominators and denominators, roughly increasing
    exponentially in $n$.
    
    A {\tt python} implementation of the algorithm (which also includes code
    for producing input for the algorithm) can be found on the author's homepage.
    It works for  valuated involutary division rings as well (implementation
    of such example is included). Depending on the computer,
    the worst-case running time for $10\times 10$ matrices over $\Q$ is about $90$ seconds. 
    Of course, a large factor of this can be saved by moving to a lower-level
    programming language.
\end{myremark}

\subsection*{Acknowledgements}

I deeply thank Eva Bayer-Fl\"{u}ckiger for introducing me with the Grothendieck-Serre conjecture, and for
hosting me in EPFL during the research.

\bibliographystyle{plain}
\bibliography{MyBib}

\end{document}